\documentclass{amsart}

\pdfoutput=1

\usepackage{amsmath}
\usepackage{amssymb}
\usepackage{bbm}
\usepackage{enumerate}
\usepackage{booktabs} 
\usepackage{appendix}
\usepackage{mathrsfs}
\usepackage{hyperref}
\usepackage{graphicx, color}
\hypersetup{
  pdftitle={Stochastic completeness for landmark space},
  pdfauthor={Karen Habermann, Stefan Sommer},
}

\usepackage[T1]{fontenc}
\usepackage[utf8]{inputenc}

\usepackage{xcolor}

\numberwithin{equation}{section}

\newcommand{\R}{\mathbb{R}}

\newcommand{\e}{\operatorname{e}}
\newcommand{\dd}{\,{\mathrm d}}
\newcommand{\db}{{\mathrm d}}

\newcommand{\im}{\operatorname{i}}

\newcommand{\eps}{\varepsilon}

\newtheorem{lemma}{Lemma}[section]

\newtheorem{thm}[lemma]{Theorem}
\newtheorem{cor}[lemma]{Corollary}

\newtheorem{remark0}[lemma]{Remark}
\newtheorem{eg0}[lemma]{Example}

\newcommand{\lm}{\mathbf{x}}
\newcommand{\lmz}{\mathbf{z}}

\newcommand{\vol}{\operatorname{vol}}
\newcommand{\Diff}{\mathrm{Diff}}
\newcommand{\Land}{\mathrm{Land}^n}

\newenvironment{remark}{\begin{remark0}\rm}{\end{remark0}}

\author[K. Habermann]{Karen Habermann}
\address{Karen Habermann, Department of Statistics, University of Warwick, Coventry, CV4 7AL, United Kingdom.}
\email{karen.habermann@warwick.ac.uk}

\author[S. Sommer]{Stefan Sommer}
\address{Stefan Sommer, Department of Computer Science, University of Copenhagen, Universitetsparken 5, DK-2100 Copenhagen E, Denmark.}
\email{sommer@di.ku.dk}

\title{Stochastic completeness for landmark space}
\begin{document}
\begin{abstract}
We study stochastic completeness for landmark spaces equipped with Riemannian metrics induced by right-invariant metrics on subgroups of the diffeomorphism group of the shape domain. We extend a previous stochastic completeness result, which only covers the case of exactly two landmarks, to landmark spaces with any number of landmarks. This succeeds the characterization of geodesic completeness for landmark spaces with arbitrary numbers of landmarks, and thus finishes the completeness characterization for landmark spaces by covering the stochastic case. The proof makes use of Grigor'yan's volume growth criterion for stochastic completeness, which requires a suitable upper bound for the volume of growing geodesic balls. We obtain quantitative controls for geodesic balls in the landmark space by bounding both its Euclidean size and the rate at which pairwise landmark distances can approach zero. We then combine this with a lower bound on the minimal eigenvalue of the landmark cometric in terms of the Fourier transform of the kernel to yield volume growth bounds sufficient to prove stochastic completeness of landmark spaces for wide classes of kernels, including Mat\'ern kernels.
\end{abstract}

\maketitle
\thispagestyle{empty}

\section{Introduction}

\noindent
Spaces of landmarks equipped with Riemannian metrics induced from right-invariant metrics on subgroups of the diffeomorphism group were introduced in~\cite{joshiLandmarkMatchingLarge2000} and play a central role in shape analysis, computational anatomy and stochastic models of shape evolution. For configurations of landmarks in $\R^d$, the metric is encoded by a positive-definite symmetric kernel on $\R^d$. For instance, for Sobolev metrics on subgroups of the diffeomorphism group, this kernel is of Mat\'ern type. The resulting Riemannian geometry couples all landmarks and becomes singular only at landmark collision, so the small-scale behavior of the kernel near the diagonals governs both geodesic and stochastic completeness properties of the landmark manifold.

For stochastic models of shape evolution, determining stochastic completeness is important to ensure that Brownian motion exists almost surely for all times, in this case for landmark shape spaces. In~\cite{habermannLongtimeExistenceBrownian2024}, the question concerning long-time existence of Brownian motion was completely answered for configurations of exactly two landmarks. In particular, for Sobolev metrics one finds the threshold phenomenon that low-order Sobolev metrics allow landmark collision in finite time with positive probability, whereas sufficiently regular Sobolev metrics and Gaussian kernels give long-time existence of Brownian motion. More recently, \cite{habermannCharacterizationGeodesicCompleteness2025a} established a full characterization of geodesic completeness for landmark spaces with smooth translation- and rotation-invariant metrics for arbitrary numbers of landmarks, and gave examples of geodesically complete but stochastically incomplete landmark manifolds.

The purpose of the present paper is to extend the stochastic completeness theory for landmark spaces from two landmarks to arbitrary numbers of landmarks. For translationally and rotationally invariant kernels, we link stochastic completeness for any number of landmarks to specific properties of the Fourier transform of the kernels. This significantly strengthens the results of \cite{habermannLongtimeExistenceBrownian2024} because instead of relying on a special reduction available for two landmarks, we give a direct global argument on the full landmark space. Our proof combines Grigor'yan's volume growth criterion with quantitative controls in terms of the Euclidean distance for geodesic balls in the landmark manifold. We show that in geodesic balls, one can control both the Euclidean size of individual landmarks and the rate at which pairwise landmark distances may approach zero. We then use a lower bound on the minimal eigenvalue of the landmark cometric to upper bound the Riemannian volume growth of geodesic balls.

The total number of landmarks affects the size of the cometric matrix, but the mechanism that prevents excessive volume growth is still driven by kernel asymptotics when the minimum distance between landmarks becomes small. We thus show that the threshold for stochastic completeness observed for Sobolev metrics in the two-landmark case persists for any number of landmarks.

We denote the configuration space of $n\geq 2$ ordered distinct landmarks in $\R^d$, for $d\geq 1$, by
\begin{displaymath}
    \Land=\{\lm=(x_1,\dots,x_n):x_1,\dots,x_n\in\R^d\text{ with }x_i\not=x_j\text{ for }i\not=j\},
\end{displaymath}
where $x_i\in\R^d$ gives the position of the $i$th landmark and where $x_i\not= x_j$ for $i,j\in\{1,\dots,n\}$ with $i\not =j$ reflects the condition that the landmarks are mutually distinct. We observe that $\Land$ can be considered as an open subset of the Euclidean space $\R^{nd}$.

As derived in~\cite{MMM}, the cometric of a Riemannian metric $g$ on $\Land$ induced by a right-invariant metric on a subgroup of the diffeomorphism group takes the form
\begin{equation}\label{eq:metrickernel}
    g^{ij}(\lm)=K(x_i,x_j),
\end{equation}
for $i,j\in\{1,\dots,n\}$ and a positive-definite symmetric kernel $K\colon\R^d\times\R^d\to\R^{d\times d}$. Throughout, we assume that the kernel $K$ is invariant under both translations and rotations. Translation invariance implies that $K\colon\R^d\times\R^d\to\R^{d\times d}$ only depends on the difference $x_i-x_j$ of the two inputs $x_i$ and $x_j$ whereas rotation invariance means that
\begin{displaymath}
    K(Ax_i,Ax_j)=AK(x_i,x_j)A^{-1}
\end{displaymath}
for all rotations $A\in SO(d)$. In particular, while the cometric kernel is in general matrix-valued, a translationally and rotationally invariant kernel $K$ is uniquely characterized by a continuous radial scalar function $k\colon [0,\infty)\to\R$ such that
\begin{equation}\label{eq:invariances}
    K(x_i,x_j)=k(\|x_i-x_j\|)I_d
\end{equation}
in terms of the Euclidean norm $\|\cdot\|$ on $\R^d$ and the $d\times d$ identity matrix $I_d$. We further assume that $k\colon [0,\infty)\to\R$ is positive and strictly decreasing, which for kernels that are positive-definite in any dimension is guaranteed by~\cite{schoenberg}, also see~\cite[Proposition~9.14]{younesShapesDiffeomorphisms2010}.

Our criterion for long-time existence of Brownian motion on the Riemannian landmark manifold $(\Land,g)$, that is, for stochastic completeness of $(\Land,g)$, crucially depends on high-frequency behavior of the Fourier transform on $\R^d$ of the function given by $x\mapsto k(\|x\|)$ for $x\in\R^d$. Specifically, if it exists, we write $\widehat{k}\colon\R^d\to\R$ for the Fourier transform
\begin{displaymath}
    \widehat{k}(\xi)=\int_{\R^d} \e^{-\im \langle x,\xi\rangle}k(\|x\|)\dd x.
\end{displaymath}
For kernels that are positive-definite in any dimension, it is a consequence of the characterization for scalar radial kernels obtained in~\cite{schoenberg} that $\|\xi\|\mapsto \widehat{k}(\|\xi\|)$ is a decreasing function on $(0,\infty)$.

We establish the following broad criterion for stochastic completeness of the landmark manifold $(\Land,g)$ for all $n\geq 2$ and $d\geq 1$ in terms of the cometric kernel. While the landmark configuration space $\Land$ is not connected for $d=1$, each connected component is stochastically complete under the assumed conditions, and thus, Brownian motion started in one connected component remains in that component for all times.
\begin{thm}\label{thm:main}
    Suppose that the radial scalar function $k\colon [0,\infty)\to\R$ is smooth on $(0,\infty)$ and that it possesses a positive continuous Fourier transform $\widehat{k}\colon\R^d\to\R$. In addition, assume that there exist $\eps>0$ and constants $b,c>0$ as well as $p\geq 1$ such that, for all $r\in(0,\eps)$,
    \begin{displaymath}
        k(0)-k(r)\leq -cr^2\log(r)
    \end{displaymath}
    and, for all sufficiently high frequencies $\xi\in\R^d$,
    \begin{displaymath}
        \widehat{k}(\xi)\geq b\|\xi\|^{-p}.
    \end{displaymath}
    Then the landmark manifold $(\Land,g)$ is stochastically complete for all $n\geq 2$.
\end{thm}

We remark that the high-frequency decay control for the Fourier transform $\widehat{k}$ prevents the even extension to $\R$ of the radial scalar function $k\colon [0,\infty)\to\R$ to be smooth at zero.

The assumed condition that $k(0)-k(r)\leq - cr^2\log(r)$ near zero is needed to control the rate at which pairwise distances between distinct landmarks may approach zero. At the same time, as a bound polynomial in $r$ and $\log(r)$, it cannot be relaxed for a broad stochastic completeness criterion because the geodesically complete yet stochastically incomplete landmark manifolds constructed in~\cite{habermannCharacterizationGeodesicCompleteness2025a} use kernels that near zero satisfy $k(0)-k(r)=r^2(1-\log(r))^\beta$ for $\beta\in(1,2]$. However, it is instead possible to assume an alternative regularity condition on the even extension to $\R$ of the radial scalar function $k$, which results in a slightly weaker criterion for stochastic completeness.
\begin{cor}\label{cor:C2}
    Suppose that the radial scalar function $k\colon [0,\infty)\to\R$ is smooth on $(0,\infty)$ and that its even extension to $\R$ is twice continuously differentiable at zero. In addition, assume that $k$ possesses a positive continuous Fourier transform $\widehat{k}\colon\R^d\to\R$ and that there exist $p\geq 1$ as well as a constant $b>0$ such that, for all sufficiently high frequencies $\xi\in\R^d$,
    \begin{displaymath}
        \widehat{k}(\xi)\geq b\|\xi\|^{-p}.
    \end{displaymath}
    Then the landmark manifold $(\Land,g)$ is stochastically complete for all $n\geq 2$.
\end{cor}

An important class of translationally and rotationally invariant kernels $K\colon\R^d\times\R^d\to\R^{d\times d}$ used in applications
arises from Sobolev operators, in the sense that the integral operator with kernel $K$ is the inverse of the Sobolev differential operator $L_{H^s}=(\operatorname{Id}-\sigma^2\Delta)^s$ for $s>\frac{d}{2}$ and $\sigma>0$. The associated radial scalar function is the Mat\'ern kernel given by, with $\nu=s-\frac{d}{2}>0$,
\begin{displaymath}
    k(r)=C_{d,s,\sigma}\left(\frac{r}{\sigma}\right)^\nu K_\nu\left(\frac{r}{\sigma}\right),
\end{displaymath}
where $K_\nu$ denotes the modified Bessel function of the second kind and $C_{d,s,\sigma}>0$ is a normalizing constant. More details are provided and discussed in Section~\ref{sec:kernels}. When applied specifically to the Mat\'ern kernels, Theorem~\ref{thm:main} gives the following characterization, which yields the same threshold previously observed in~\cite{habermannLongtimeExistenceBrownian2024} for the two-landmark case.
\begin{cor}\label{cor:sobolev}
    Suppose the Riemannian metric $g$ on the landmark configuration space $\Land$ arises from the Sobolev operator $L_{H^s}$. If we have $s\geq 1+\frac{d}{2}$ then the landmark manifold $(\Land,g)$ is stochastically complete for all $n\geq 2$.
\end{cor}

We will see that for the critical case $\nu=s-\frac{d}{2}=1$, we are exactly in a situation where $k(0)-k(r)$ behaves like $-cr^2\log(r)$ near zero and where Corollary~\ref{cor:C2} would not apply. We further remark that for $\nu\in (0,1)$, the criterion derived in~\cite{habermannLongtimeExistenceBrownian2024} shows that on the landmark space $(\Land,g)$ with $n=2$, the two landmarks collide with positive probability in finite time.

This threshold observed in the stochastic completeness for cometric kernels arising from Sobolev operators is exactly the same one as for the geodesic completeness, where it is a consequence of~\cite{bauerOverviewGeometriesShape2014} or~\cite{habermannCharacterizationGeodesicCompleteness2025a} that, for any number of landmarks, the manifold $(\Land,g)$ is geodesically complete if and only if $s\geq 1+\frac{d}{2}$. On the level of the infinite-dimensional Sobolev diffeomorphism groups, geodesic completeness is known for $s>1+\frac{d}{2}$, see~\cite{sobolevcomplete1,sobolevcomplete2}, with very limited results available in the critical case $s=1+\frac{d}{2}$ such as~\cite{criticalsobolev1,criticalsobolev2}.

In the proof of Theorem~\ref{thm:main}, we make use of the following result by Grigor'yan, see~\cite[Theorem~1]{grigoryan1} and \cite[Theorem~9.1]{grigoryan2}, which states that a geodesically complete Riemannian manifold subject to a suitable volume growth control for growing geodesic balls is also stochastically complete.
\begin{thm}[Grigor'yan~\cite{grigoryan1,grigoryan2}]\label{thm:grigoryan}
    Let $(M,g)$ be a geodesically complete Riemannian manifold. Suppose that, for some $z\in M$, the volumes $V(r)$ of the geodesic balls with center $z$ and of radius $r>0$ satisfy, for some $a>0$,
    \begin{displaymath}
        \int_a^\infty\frac{r\dd r}{\log V(r)} =\infty.
    \end{displaymath}
    Then $(M,g)$ is stochastically complete.
\end{thm}

The criterion by Grigor'yan was successfully applied in~\cite{curve_sto_complete} for spaces of discrete regular curves to prove that also there sufficiently regular Sobolev-type metrics ensure stochastic completeness.

As detailed later, under the conditions imposed in Theorem~\ref{thm:main}, the geodesic completeness of the landmark manifold $(\Land,g)$ is implied by the characterization derived in~\cite{habermannCharacterizationGeodesicCompleteness2025a} that the landmark manifold $(\Land,g)$ is geodesically complete if and only if, for all $a>0$,
\begin{displaymath}
    \int_0^a \frac{\db r}{\sqrt{k(0)-k(r)}}=\infty.
\end{displaymath}
The core of the argument therefore lies in obtaining a sufficiently good volume growth control for growing geodesic balls that allows us to apply the criterion by Grigor'yan. This is achieved by controlling both the Euclidean norm of individual landmarks and their pairwise distances on geodesic balls, and by lower bounding the minimal eigenvalue of the cometric of the Riemannian metric $g$ in terms of the minimal pairwise landmark distance.

\medskip

\paragraph{\bf Paper outline.}
In Section~\ref{sec:kernels}, we first present an overview of cometric kernels and their Fourier transforms, with a focus on Mat\'ern kernels arising from Sobolev differential operators. We then proceed in Section~\ref{sec:distancecontrol} to derive the required controls for Euclidean norms and pairwise distances on geodesic balls, with Lemma~\ref{lem:controlnorm} providing an upper bound for the Euclidean norm of individual landmarks and Lemma~\ref{lem:controlseparationgeneral} giving a general lower bound for the pairwise distances of landmarks. The explicit lower bound obtained for the minimal pairwise distance under the assumed behavior of the kernel near zero is stated in Lemma~\ref{lem:controlseparation}. In Section~\ref{sec:mineval}, we establish an upper bound for the volume form on geodesic balls by exploiting a lower bound on the minimal eigenvalue of the cometric in terms of the minimal pairwise distance. We finally put everything together in Section~\ref{sec:stocomplete} to prove Theorem~\ref{thm:main}, Corollary~\ref{cor:C2} and Corollary~\ref{cor:sobolev}.

\medskip

\paragraph{\bf Acknowledgements.}
The work presented in this paper was supported by the Villum Foundation Grant 40582, the Novo Nordisk Foundation grants NNF18OC0052000, NNF24OC0093490 and NNF24OC0089608, and by the London Mathematical Society through a Computer Science Small Grant (Scheme 7).

\section{Cometric kernels and their Fourier transforms}
\label{sec:kernels}

\noindent
We give a more detailed overview of cometric kernels on shape spaces and discuss properties of their Fourier transforms, with a focus on the Mat\'ern kernels that arise from Sobolev differential operators.

The large deformation diffeomorphic metric mapping framework employed in shape analysis characterizes shape variations as diffeomorphic deformations of the ambient space in which the shapes reside, see, e.g.,~\cite{trouveDiffeomorphismsGroupsPattern1998,younesComputableElasticDistances1998,younesShapesDiffeomorphisms2010}. It uses a right-invariant Riemannian metric on a suitable subgroup of the diffeomorphism group of the shape domain to quantify shape changes. A frequently considered subgroup of the diffeomorphism group $\Diff(\R^d)$ for $\R^d$ is the infinite-dimensional Lie group $\Diff_c(\R^d)$ of compactly supported diffeomorphisms. The right-invariant Riemannian metric then arises by extending an inner product for which point evaluations are norm-continuous on the space ${\mathfrak X}_c(\R^d)$ of compactly supported vector fields on $\R^d$. The completion of ${\mathfrak X}_c(\R^d)$ with respect to this inner product is a Hilbert space with a positive reproducing kernel $K\colon\R^d\times\R^d\to\R^{d\times d}$.

The metric on the subgroup $\Diff_c(\R^d)$ descends to the landmark configuration space $\Land$ due to the assumed right-invariance. A compactly supported diffeomorphism $\phi\in\Diff_c(\R^d)$ acts on a landmark configuration $\lm\in \Land$ from the left as $\phi.\lm=(\phi(x_1),\dots,\phi(x_n))$. For $\lm\in \Land$ fixed, this action yields a submersion from $\Diff_c(\R^d)$ to $\Land$. The right-invariance of the metric on $\Diff_c(\R^d)$ then implies that there exists a unique Riemannian metric $g$ on $\Land$ such that the submersion from $\Diff_c(\R^d)$ to $\Land$ is Riemannian, which is exactly the metric whose cometric is given by~\eqref{eq:metrickernel}, see~\cite{MMM} for further details.

If the inner product on ${\mathfrak X}_c(\R^d)$ is induced by the Sobolev differential operator
\begin{displaymath}
    L_{H^s}=\left(\operatorname{Id}-\sigma^2\Delta\right)^s
\end{displaymath}
for $s>\frac{d}{2}$ and $\sigma>0$, we determine the positive reproducing kernel $K$, whose associated integral operator is the inverse of $L_{H^s}$, by working in the Fourier domain.

For a function $f\colon\R^d\to\R$, its Fourier transform $\widehat{f}\colon\R^d\to\R$ on $\R^d$, if it exists, is defined by, for $\xi\in\R^d$,
\begin{displaymath}
    \widehat{f}(\xi)=\int_{\R^d} \e^{-\im \langle x,\xi\rangle}f(x)\dd x,
\end{displaymath}
with the inverse given by, for $x\in\R^d$,
\begin{displaymath}
    f(x)=\frac{1}{(2\pi)^{d}}\int_{\R^d} \e^{\im \langle x,\xi\rangle}\widehat{f}(\xi)\dd \xi.
\end{displaymath}
The Sobolev differential operator $L_{H^s}$ now has the Fourier symbol, for $\xi\in\R^d$,
\begin{displaymath}
    \widehat{L_{H^s}v}(\xi)=\left(1+\sigma^2\|\xi\|^2\right)^s\widehat{v}(\xi)
\end{displaymath}
applied componentwise to vector fields. The corresponding Green's kernel is therefore scalar and invariant under both translations and rotations. Since its Fourier transform $\widehat{k}\colon\R^d\to\R$ on $\R^d$ has the form, for $\xi\in\R^d$,
\begin{equation}\label{eq:maternfourier}
    \widehat{k}(\xi)=\left(1+\sigma^2\|\xi\|^2\right)^{-s},
\end{equation}
computing the inverse Fourier transform shows that the associated radial scalar function is indeed the Mat\'ern kernel, with $r=\|x\|$ and $\nu=s-\frac d2>0$,
\begin{displaymath}
    k(r)=C_{d,s,\sigma}\left(\frac{r}{\sigma}\right)^\nu K_\nu\left(\frac{r}{\sigma}\right).
\end{displaymath}
For Mat\'ern kernels, the decay of $\widehat{k}$ at high frequency is of order $\|\xi\|^{-2s}$, while the behavior of the radial scalar function $k$ near zero is encoded by the modified Bessel function $K_\nu$ of the second kind. We have the asymptotics, as $r\downarrow 0$,
\begin{equation}\label{eq:bessel-asymptotics}
    r^\nu K_{\nu}(r)
    =
    2^{\nu-1}\Gamma(\nu)-
    \begin{cases}
        -2^{\nu-1}\Gamma(-\nu)r^{2\nu}+o\left(r^{2\nu}\right),
        &\nu\in(0,1)\;,\\
        -2^{-1}r^2\log(r)+o\left(r^2\log(r)\right),
        &\nu=1\;,\\
        2^{\nu-3}\Gamma(\nu-1)r^2+o\left(r^2\right),
        &\nu\in (1,\infty)\;,
    \end{cases}
\end{equation}
which follow from \cite[Equations (6.1.15-17) and (9.6.2--11)]{abramowitz1972handbook}.

\section{Landmark norms and pairwise distances on geodesic balls}
\label{sec:distancecontrol}

\noindent
We prove an upper bound for the Euclidean norm of individual landmarks and a lower bound for the pairwise distances of landmarks on geodesic balls in the landmark manifold. In particular, we will see that geodesic balls of growing radii neither grow too rapidly nor approach the collision set too rapidly if measured with respect to the Euclidean metric.

A key relation used in proving these bounds is that, for a landmark configuration $\lm\in \Land$, a vector $v\in T_\lm \Land$ and a covector $\alpha\in T_\lm^\star \Land$, we have that
\begin{equation}\label{eq:metricCS}
    \alpha(v)^2\leq g(v,v)\, g^{-1}(\alpha,\alpha).
\end{equation}
This is true in a general Riemannian manifold and follows from the Cauchy--Schwarz inequality after applying the musical isomorphism sharp induced by $g$. The inequality~\eqref{eq:metricCS} is also exploited in~\cite{habermannCharacterizationGeodesicCompleteness2025a} to bound the lengths of curves on landmark spaces in terms of integrals involving the radial scalar function $k$.

We denote the open geodesic ball in the landmark manifold $(\Land,g)$ with center $\lmz\in \Land$ and of radius $R>0$ by $B(\lmz,R)$, which is the set of all $\lm\in \Land$ that can be connected to $\lmz$ through a curve of length less than $R$. The following lemma provides a control for the Euclidean norm of individual landmarks in a landmark configuration $\lm\in B(\lmz,R)$. Throughout, we have $\lm=(x_1,\dots,x_n)$ and $\lmz=(z_1,\dots,z_n)$ if we need to refer to individual landmarks in $\R^d$.

\begin{lemma}\label{lem:controlnorm}
    Suppose the radial scalar function $k\colon [0,\infty)\to\R$ is smooth on $(0,\infty)$. Fix $\lmz\in \Land$ and let $R>0$. We then have, for all $\lm\in B(\lmz,R)$ and all $i\in\{1,\dots, n\}$,
    \begin{displaymath}
        \|x_i\|\leq \sqrt{k(0)}R+\|z_{i}\|.
    \end{displaymath}
\end{lemma}
\begin{proof}
    For $\delta>0$ and $i\in\{1,\dots,n\}$, we define $f_i^{\delta}\colon \Land\to \R$ by
    \begin{displaymath}
        f_i^{\delta}(\lm)=\sqrt{\|x_i\|^2+\delta}.
    \end{displaymath}
    Using that the cometric is given by~\eqref{eq:metrickernel} and \eqref{eq:invariances}, and that, for all $\lm\in \Land$,
    \begin{displaymath}
        \left(\db f_i^\delta\right)(\lm)=\frac{\langle x_i,\db x_i\rangle}{f_i^{\delta}(\lm)},
    \end{displaymath}
    we obtain
    \begin{displaymath}
        g^{-1}(\lm)\left(\db f_i^\delta,\db f_i^\delta\right)
        =\frac{1}{\left(f_i^{\delta}(\lm)\right)^2}\left(k(0) \|x_i\|^2\right).
    \end{displaymath}
    Since $\delta>0$, it follows that
    \begin{equation}\label{eq:cometricbound}
        0 < g^{-1}\left(\db f_i^\delta,\db f_i^\delta\right) < k(0).
    \end{equation}

    In the next step, we use that due to $\lm\in B(\lmz,R)$ there exists a piecewise differentiable curve $\gamma\colon[0,T]\to \Land$ for some $T>0$ such that $\gamma(0)=\lmz$ and $\gamma(T)=\lm$ of length less than $R$, that is,
    \begin{equation}\label{eq:curvebound}
        \int_0^T\sqrt{g(\dot{\gamma}(t),\dot{\gamma}(t))}\dd t < R,
    \end{equation}
    where we write $\dot{\gamma}=\frac{\db\gamma}{\db t}$. Applying~\eqref{eq:metricCS} and \eqref{eq:cometricbound} along the curve $\gamma$, we deduce that, for almost all $t\in[0,T]$,
    \begin{displaymath}
        \left|\frac{\db}{\db t}\left(f_i^{\delta}\circ\gamma\right)(t)\right|\leq
        \sqrt{g(\dot{\gamma}(t),\dot{\gamma}(t))}\sqrt{g^{-1}\left(\db f_i^\delta,\db f_i^\delta\right)}
        \leq \sqrt{k(0)}\sqrt{g(\dot{\gamma}(t),\dot{\gamma}(t))}.
    \end{displaymath}
    Integrating this inequality from $0$ to $T$ and using the change of variables $r=f_i^{\delta}(\gamma(t))$, we conclude
    \begin{displaymath}
        f_i^{\delta}(\lm)-f_i^{\delta}(\lmz)
        \leq \sqrt{k(0)}\int_0^T\sqrt{g(\dot{\gamma}(t),\dot{\gamma}(t))}\dd t.
    \end{displaymath}
    Note that the latter indeed remains true even if $f_i^{\delta}(\lm)<f_i^{\delta}(\lmz)$. As $\delta>0$ was arbitrary, the claimed result follows from~\eqref{eq:curvebound}.
\end{proof}

We proceed by analyzing pairwise distances between landmarks in landmark configurations on geodesic balls. The following general result is subsequently used to derive an explicit lower bound for the pairwise distances of landmarks under the additional assumption that $k(0)-k(r)\leq - cr^2\log(r)$ near zero.
\begin{lemma}\label{lem:controlseparationgeneral}
    Suppose the radial scalar function $k\colon [0,\infty)\to\R$ is smooth on $(0,\infty)$. Fix $\lmz\in \Land$ and let $R>0$. We then have, for all $\lm\in B(\lmz,R)$ and all $i,j\in\{1,\dots,n\}$ with $i\not=j$,
    \begin{displaymath}
        \int_{\|x_i-x_j\|}^{\|z_i-z_j\|}
        \frac{\db r}{\sqrt{2\left(k(0)-k(r)\right)}} < R.
    \end{displaymath}
\end{lemma}
\begin{proof}
    We let $f_{ij}\colon \Land\to\R$, for $i,j\in\{1,\dots,n\}$ with $i\not=j$, be given by
    \begin{displaymath}
        f_{ij}(\lm)=\|x_i-x_j\|.
    \end{displaymath}
    For all $\lm\in \Land$, we have $f_{ij}(\lm)\not=0$ and
    \begin{displaymath}
        \left(\db f_{ij}\right)(\lm)
        =\frac{\langle x_i-x_j,\db x_i-\db x_j\rangle}{f_{ij}(\lm)}.
    \end{displaymath}
    Together with~\eqref{eq:metrickernel} and \eqref{eq:invariances}, we obtain
    \begin{displaymath}
        g^{-1}(\lm)(\db f_{ij},\db f_{ij})
        =\frac{1}{\left(f_{ij}(\lm)\right)^2}
        \left(k(0)\|x_i-x_j\|^2-2k(\|x_i-x_j\|)\|x_i-x_j\|^2+k(0)\|x_i-x_j\|^2\right),
    \end{displaymath}
    that is,
    \begin{equation}\label{eq:sepdistance}
        g^{-1}(\lm)(\db f_{ij},\db f_{ij})
        =2\left(k(0)-k\left(f_{ij}(\lm)\right)\right).
    \end{equation}

    Since $\lm\in B(\lmz,R)$, there exists a piecewise differentiable curve $\gamma\colon[0,T]\to \Land$ for some $T>0$ such that $\gamma(0)=\lm$ and $\gamma(T)=\lmz$ of length less than $R$. We write $\dot{\gamma}=\frac{\db\gamma}{\db t}$. From~\eqref{eq:metricCS} and \eqref{eq:sepdistance}, it follows that, for almost all $t\in[0,T]$,
    \begin{displaymath}
        \frac{\left|\frac{\db}{\db t} \left(f_{ij}\circ \gamma\right)(t)\right|}{\sqrt{2\left(k(0)-k\left(f_{ij}\left(\gamma(t)\right)\right)\right)}}
        \leq \sqrt{g\left(\dot{\gamma}(t),\dot{\gamma}(t)\right)}.
    \end{displaymath}
    Integrating this inequality from $0$ to $T$ and using the change of variables $r=f_{ij}(\gamma(t))$, we deduce
    \begin{displaymath}
        \int_{\|x_i-x_j\|}^{\|z_i-z_j\|}
        \frac{\db r}{\sqrt{2\left(k(0)-k(r)\right)}}
        \leq \int_0^T \frac{\left|\frac{\db}{\db t} \left(f_{ij}\circ \gamma\right)(t)\right|}{\sqrt{2\left(k(0)-k\left(f_{ij}\left(\gamma(t)\right)\right)\right)}}\dd t
        \leq \int_0^T\sqrt{g\left(\dot{\gamma}(t),\dot{\gamma}(t)\right)}\dd t < R,
    \end{displaymath}
    as required.    
\end{proof}

\begin{remark}
    We note that in Lemma~\ref{lem:controlseparationgeneral}, it is not necessary to assume geodesic completeness of the landmark manifold $(\Land,g)$. According to~\cite{habermannCharacterizationGeodesicCompleteness2025a}, under translational and rotational invariance, geodesic completeness is equivalent to, for $a>0$,
    \begin{displaymath}
        \int_0^a\frac{\db r}{\sqrt{k(0)-k(r)}}=\infty.
    \end{displaymath}
    In particular, on a geodesically incomplete landmark manifold and for sufficiently large $R>0$, the inequality stated in Lemma~\ref{lem:controlseparationgeneral} is automatically true whenever $\|x_i-x_j\|\leq \|z_i-z_j\|$, even when the distance $\|x_i-x_j\|$ becomes arbitrarily small.

    At the same time, for a geodesically complete landmark manifold, it is exactly Lemma~\ref{lem:controlseparationgeneral} that allows us to derive a lower bound for the pairwise distances between landmarks in landmark configurations on geodesic balls.
\end{remark}

Following~\cite{wendland}, we denote the separation distance of a landmark configuration $\lm\in \Land$ by $q(\lm)$, which is defined by
\begin{displaymath}
    q(\lm)=\frac{1}{2}\min_{i\not= j} \|x_i-x_j\|.
\end{displaymath}
It gives the maximum radius $r>0$ such that the open Euclidean balls in $\R^d$ of radius $r$ and centered at the individual landmarks are all disjoint.

The next lemma provides a lower bound for the separation distance of landmark configurations on geodesic balls under the assumption that $k(0)-k(r)\leq -cr^2 \log(r)$ near zero, which is one of the assumptions in Theorem~\ref{thm:main}.
\begin{lemma}\label{lem:controlseparation}
    Suppose that the radial scalar function $k\colon [0,\infty)\to\R$ is smooth on $(0,\infty)$ and that there exist $\eps>0$ as well as a constant $c>0$ such that $k(0)-k(r)\leq -cr^2\log(r)$ for all $r\in(0,\eps)$. Fix $\lmz\in \Land$. For sufficiently large $R>0$, we then have, for all $\lm\in B(\lmz, R)$,
    \begin{displaymath}
        q(\lm)\geq \frac{1}{2}\exp\left(-2cR^2\right).
    \end{displaymath}
\end{lemma}
\begin{proof}
    By definition of the separation distance, we need to show that, for sufficiently large $R>0$, we have, for all $\lm\in B(\lmz, R)$ and for all $i,j\in\{1,\dots,n\}$ with $i\not= j$,
    \begin{equation}\label{eq:pairwisedistbound}
        \|x_i-x_j\|\geq\exp\left(-2c R^2\right).
    \end{equation}
    If $\|x_i-x_j\|\geq\eps$ this is certainly true for sufficiently large $R$.

    We now consider the case where $\|x_i-x_j\|<\eps$ for $\lm\in B(\lmz,R)$ and $i,j\in\{1,\dots,n\}$ with $i\not= j$.
    Under the assumptions on the radial scalar function $k$ we have, for all $r\in(0,\eps)$,
    \begin{equation}\label{eq:kernelintbound}
        \frac{1}{\sqrt{-2cr^2\log(r)}}\leq \frac{1}{\sqrt{2(k(0)-k(r))}}.
    \end{equation}
    The left-hand side of the inequality is well-defined because $\log(r)<0$ for $r\in(0,1)$ and the bound on the radial scalar function $k$ can only be true for some $\eps\in(0,1)$. We compute, for $a\in(0,\eps)$,
    \begin{displaymath}
        \int_a^\eps \frac{\db r}{\sqrt{-2cr^2\log(r)}}
        =\sqrt{\frac{2}{c}}\left(\sqrt{-\log(a)}-\sqrt{-\log(\eps)}\right),
    \end{displaymath}
    which together with~\eqref{eq:kernelintbound} implies that, for $a\in(0,\eps)$,
    \begin{equation}\label{eq:intcontrolnearzero}
        \sqrt{\frac{2}{c}}\left(\sqrt{-\log(a)}-\sqrt{-\log(\eps)}\right)
        \leq \int_a^\eps \frac{\db r}{\sqrt{2(k(0)-k(r))}}.
    \end{equation}
    Set
    \begin{equation}\label{eq:const}
        C_{ij}
        =\int_\eps^{\|z_i-z_j\|}\frac{\db r}{\sqrt{2(k(0)-k(r))}}
        -\sqrt{\frac{2}{c}}\sqrt{-\log(\eps)}.
    \end{equation}
    It is important to note that the constant $C_{ij}$ may be negative.
    By Lemma~\ref{lem:controlseparationgeneral}, we have
    \begin{displaymath}
        \int_{\|x_i-x_j\|}^{\eps}
        \frac{\db r}{\sqrt{2\left(k(0)-k(r)\right)}}
        +\int_{\eps}^{\|z_i-z_j\|}
        \frac{\db r}{\sqrt{2\left(k(0)-k(r)\right)}}
        =\int_{\|x_i-x_j\|}^{\|z_i-z_j\|}
        \frac{\db r}{\sqrt{2\left(k(0)-k(r)\right)}} \leq R.
    \end{displaymath}
    Using~\eqref{eq:intcontrolnearzero} with $a=\|x_i-x_j\|$ as well as~\eqref{eq:const}, we obtain
    \begin{displaymath}
        \sqrt{\frac{2}{c}}\sqrt{-\log\left(\|x_i-x_j\|\right)}+C_{ij}\leq R.
    \end{displaymath}
    If we now choose $R>0$ large enough such that $R>-C_{ij}$ then
    \begin{displaymath}
        \sqrt{\frac{2}{c}}\sqrt{-\log\left(\|x_i-x_j\|\right)}\leq 2R.
    \end{displaymath}
    As the left-hand side of the latest inequality is guaranteed to be positive, we can square both sides to deduce
    \begin{displaymath}
        -\log\left(\|x_i-x_j\|\right)\leq 2cR^2,
    \end{displaymath}
    which in turn shows that
    \begin{displaymath}
        \|x_i-x_j\|\geq \exp\left(-2cR^2\right).
    \end{displaymath}
    
    Since there are only finitely many $i,j\in\{1,\dots,n\}$ with $i\not= j$, we can thus indeed choose $R>0$ large enough such that~\eqref{eq:pairwisedistbound} is true for all non-trivial pairwise distances between landmarks, thereby proving the claimed result.
\end{proof}

\section{Control for volume form on geodesic balls}
\label{sec:mineval}

\noindent
We derive an upper bound for the Riemannian volume form on geodesic balls in the landmark manifold, which subsequently yields a sufficiently good upper bound for the volume of geodesic balls. The control for the volume form is obtained by lower bounding the minimal eigenvalue of the landmark cometric in terms of the separation distance of a landmark configuration.

We first observe that, with the assumed translation and rotation invariance, the cometric given by~\eqref{eq:metrickernel} and \eqref{eq:invariances} on $\Land$ has, for a fixed radial scalar function $k\colon[0,\infty)\to\R$, the same eigenvalues for all $d\geq 1$, only their multiplicity changes accordingly. Since for $d=1$ the cometric of the Riemannian metric $g$ on $\Land$ has the same form as the interpolation matrix considered by Wendland in~\cite{wendland}, the following theorem is a consequence of~\cite[Theorem~12.3]{wendland}.

\begin{thm}\label{thm:wendland}
    Suppose that the radial scalar function $k\colon [0,\infty)\to\R$ is smooth on $(0,\infty)$ and that it possesses a positive continuous Fourier transform $\widehat{k}\colon\R^d\to\R$. For $M>0$, we set
    \begin{displaymath}
        \widehat{k}_M:=\inf_{\|\xi\|\leq 2M} \widehat{k}(\xi).
    \end{displaymath}
    Fix a landmark configuration $\lm\in \Land$. Then, for any
    \begin{displaymath}
        M\geq \frac{12}{q(\lm)}\left(\frac{\pi\left(\Gamma\left(1+\frac{d}{2}\right)\right)^2}{9}\right)^\frac{1}{1+d}
    \end{displaymath}
    or, simply,
    \begin{displaymath}
        M\geq\frac{7 d}{q(\lm)},
    \end{displaymath}
    the minimal eigenvalue $\lambda_{\min}\left(g^{-1}(\lm)\right)$ of the cometric at the landmark configuration $\lm$ satisfies
    \begin{displaymath}
        \lambda_{\min}\left(g^{-1}(\lm)\right)
        \geq \frac{\widehat{k}_M}{2\Gamma\left(1+\frac{d}{2}\right)}\left(\frac{M}{2^{3/2}}\right)^d.
    \end{displaymath}
\end{thm}

The Riemannian volume form $\vol_g$ on the landmark manifold $(\Land,g)$ can be expressed in terms of the Euclidean volume form $\db\omega$ on $\R^{nd}$ as
\begin{equation}\label{eq:volumeform}
    \vol_g=\sqrt{\det g}\dd\omega.
\end{equation}
The following result therefore provides an upper bound for the volume form on geodesic balls in the landmark manifold. It is a consequence of the explicit lower bound for the separation distance of landmark configurations in geodesic balls stated in Lemma~\ref{lem:controlseparation} and the lower bound on the minimal eigenvalue of the landmark cometric given by Theorem~\ref{thm:wendland}.

\begin{lemma}\label{lem:volumeformcontrol}
    Suppose that the radial scalar function $k\colon [0,\infty)\to\R$ is smooth on $(0,\infty)$ and that it possesses a positive continuous Fourier transform $\widehat{k}\colon\R^d\to\R$. In addition, assume that there exist $\eps>0$ and constants $b,c>0$ as well as $p\geq 1$ such that, for all $r\in(0,\eps)$,
    \begin{displaymath}
        k(0)-k(r)\leq -cr^2\log(r)
    \end{displaymath}
    and, for all sufficiently high frequencies $\xi\in\R^d$,
    \begin{displaymath}
        \widehat{k}(\xi)\geq b\|\xi\|^{-p}.
    \end{displaymath}
    Fix $\lmz\in \Land$. Then there exist constants $C_d,c_d>0$ depending on $d\geq 1$ and the regularity of the function $k$ such that, for sufficiently large $R>0$, we have, for all $\lm\in B(\lmz,R)$,
    \begin{displaymath}
        \sqrt{\det g(\lm)}\leq C_d \exp\left(nc_d R^2\right).
    \end{displaymath}
\end{lemma}
\begin{proof}
    By the assumptions on the Fourier transform $\widehat{k}\colon\R^d\to\R$, there exist $p\geq 1$ and a constant $b>0$ such that, for sufficiently large $M>0$,
    \begin{equation}\label{eq:FTbound}
        \widehat{k}_M=\inf_{\|\xi\|\leq 2M} \widehat{k}(\xi)
        \geq b (2M)^{-p}.
    \end{equation}
    Without loss of generality, we may and will suppose that $p\geq 1+d$.
    By Lemma~\ref{lem:controlseparation}, for sufficiently large $R>0$, we further have that, for all $\lm\in B(\lmz,R)$,
    \begin{displaymath}
        q(\lm)\geq \frac{1}{2}\exp\left(-2cR^2\right).
    \end{displaymath}
    Since this implies that, for all $\lm\in B(\lmz,R)$,
    \begin{displaymath}
        \frac{7d}{q(\lm)}\leq 14d \exp\left(2cR^2\right),
    \end{displaymath}
    we can apply Theorem~\ref{thm:wendland} with $M=14d \exp\left(2cR^2\right)$.
    Due to~\eqref{eq:FTbound}, this shows that there exists a constant $B_d>0$ depending on $b>0$, $d\geq 1$ and $p\geq d+1$ such that, for sufficiently large $R>0$, we have, for all $\lm\in B(\lmz,R)$,
    \begin{displaymath}
        \lambda_{\min}\left(g^{-1}(\lm)\right)
        \geq \frac{b \left(2M\right)^{-p}}{2\Gamma\left(1+\frac{d}{2}\right)}\left(\frac{M}{2^{3/2}}\right)^d
        =B_d \exp\left(-2c(p-d)R^2\right),
    \end{displaymath}
    where $p-d\geq 1$ by construction. We finally use the relation
    \begin{displaymath}
        \det g(\lm)\leq \left(\lambda_{\min}\left(g^{-1}(\lm)\right)\right)^{-nd}
    \end{displaymath}
    to deduce the claimed result.
\end{proof}

\section{Stochastic completeness arguments}
\label{sec:stocomplete}

\noindent
We conclude by combining our results in previous sections to prove Theorem~\ref{thm:main} and then arguing how Corollary~\ref{cor:C2} and Corollary~\ref{cor:sobolev} follow.

We establish Theorem~\ref{thm:main} by using the upper bound for the Euclidean norm of individual landmarks in landmark configurations on geodesic balls, see Lemma~\ref{lem:controlnorm}, and the upper bound for the volume form on geodesic balls, see Lemma~\ref{lem:volumeformcontrol}, and by applying Grigor’yan's volume growth criterion stated in Theorem~\ref{thm:grigoryan}.
\begin{proof}[Proof of Theorem~\ref{thm:main}]
    By assumption, we have $k(0)-k(r)\leq -cr^2\log(r)$ for all $r\in(0,\eps)$, which implies that, for all $a\in(0,\eps)$,
    \begin{displaymath}
        \int_0^a \frac{\db r}{\sqrt{k(0)-k(r)}}
        \geq \int_0^a\frac{\db r}{\sqrt{-cr^2\log(r)}}
        =\lim_{r\downarrow 0}\frac{2}{\sqrt{c}}\left(\sqrt{-\log(r)}-\sqrt{-\log(a)}\right)
        =\infty.
    \end{displaymath}
    Thus, by the criterion derived in~\cite{habermannCharacterizationGeodesicCompleteness2025a} and stated in the introduction, the landmark manifold $(\Land,g)$ is geodesically complete.

    Now fix $\lmz\in \Land$. By Lemma~\ref{lem:controlnorm}, for sufficiently large $R>0$, we then have, for all $\lm\in B(\lmz,R)$ and all $i\in\{1,\dots,n\}$,
    \begin{equation}\label{eq:normcontrolweak}
        \|x_i\|\leq 2\sqrt{k(0)} R.
    \end{equation}
    Let $V(\lmz,R)$ denote the volume of the geodesic ball $B(\lmz,R)$ with respect to the Riemannian volume form $\vol_g$ on $(\Land,g)$ given by~\eqref{eq:volumeform}. From Lemma~\ref{lem:volumeformcontrol} and the spherical bound~\eqref{eq:normcontrolweak}, we deduce that there exist constants $A_d,c_d>0$ depending on $d\geq 1$ and properties of the function $k$ such that, for all sufficiently large $R>0$,
    \begin{displaymath}
        V(\lmz,R)\leq A_d R^{nd} \exp\left(nc_d R^2\right).
    \end{displaymath}
    It follows that, for sufficiently large $a>0$, we have
    \begin{displaymath}
        \int_a^\infty\frac{R\dd R}{\log V(\lmz,R)}
        \geq\int_a^\infty\frac{R\dd R}{\log(A_d)+nd\log(R)+nc_d R^2}
        =\infty.
    \end{displaymath}
    Therefore, the volume growth criterion by Grigor’yan shows that the landmark manifold $(\Land,g)$ is stochastically complete for all $n\geq 2$.
\end{proof}

We see that Corollary~\ref{cor:C2} is a consequence of Theorem~\ref{thm:main} because if the even extension to $\R$ of $k\colon[0,\infty)\to\R$ is twice continuously differentiable at zero then, for some constant $c>0$ and as $r\downarrow 0$,
\begin{displaymath}
    k(r)=k(0)-\frac{1}{2}c r^2+o\left(r^2\right)\geq k(0)+\frac{1}{2}cr^2\log(r)+o\left(r^2\right),
\end{displaymath}
and so all the assumptions in Theorem~\ref{thm:main} are satisfied.

Moreover, Corollary~\ref{cor:sobolev} follows from Theorem~\ref{thm:main} due to the small-scale asymptotics for the Mat\'ern kernels being governed by~\eqref{eq:bessel-asymptotics} with $\nu=s-\frac{d}{2}\geq 1$ and the property of the positive continuous Fourier transform given in~\eqref{eq:maternfourier} that the decay at high frequency is of order $\|\xi\|^{-2s}$. We close by highlighting that for $s=1+\frac{d}{2}$, that is, with $\nu=1$, we are exactly in the small-scale regime where, for some constant $c>0$ and as $r\downarrow 0$,
\begin{displaymath}
    k(r)=k(0)+\frac{1}{2}cr^2\log(r)+o\left(r^2\log(r)\right),
\end{displaymath}
that is, where the full force of Theorem~\ref{thm:main} is needed.

\bibliographystyle{plain}
\bibliography{references}

\end{document}